\documentclass[11pt,a4paper]{article}
\usepackage[OT1]{fontenc}
\usepackage{amsmath,amssymb,amsfonts,amsthm,mathtools,mathrsfs}
\usepackage{mathptmx}
\usepackage[margin=24mm]{geometry}
\usepackage{enumitem,booktabs,array}
\usepackage[hidelinks]{hyperref}
\setlist[enumerate]{label=\textup{(\arabic*)},leftmargin=2em,itemsep=2pt}
\numberwithin{equation}{section}
\theoremstyle{definition}
\newtheorem{definition}{Definition}[section]

\theoremstyle{plain}
\newtheorem{theorem}[definition]{Theorem}
\newtheorem{proposition}[definition]{Proposition}
\newtheorem{lemma}[definition]{Lemma}
\newtheorem{corollary}[definition]{Corollary}
\newtheorem*{theoremA}{Theorem A}
\theoremstyle{remark}

\DeclareMathOperator{\Ext}{Ext}
\DeclareMathOperator{\Hom}{Hom}

\DeclareMathOperator{\pd}{pd}
\DeclareMathOperator{\depth}{depth}

\newcommand{\Z}{\mathbb Z}
\newcommand{\N}{\mathbb N}

\newcommand{\gmod}[1]{#1\text{-}\mathrm{gr}}
\newcommand{\umod}[1]{#1\text{-}\mathrm{mod}}
\newcommand{\tw}[2]{{}^{#1}\!#2}
\newcommand{\Id}{\mathrm{Id}}
\hypersetup{pdftitle={Self-extensions over quantum complete intersections},
  pdfsubject={Compact manuscript; left modules and arbitrary commutation parameters}}

\begin{document}
\begin{center}
{\Large\bfseries The Auslander--Reiten conjecture for quantum complete intersections}
\end{center}

\medskip
\centerline{\textbf{Weiheng Xia}$^*$ }

 \renewcommand{\thefootnote}{\alph{footnote}}
 \setcounter{footnote}{-1} \footnote{ $^*$ Corresponding author.
 Email:xia.weiheng@qq.com}

\begin{abstract}
Let \(A\) be a finite-dimensional quantum complete intersection over an
arbitrary field. We prove that a finite-dimensional left \(A\)-module \(M\)
is projective whenever \(\Ext_A^1(M,M)=\Ext_A^2(M,M)=0\), with no
restriction on the multiplicative orders of the commutation parameters.
Consequently, \(A\) satisfies the Auslander--Reiten conjecture and
Tachikawa's second conjecture. The proof combines the gradability of rigid modules with
successive graded twists and the self-extension criterion of
Avramov and Buchweitz for commutative complete intersections. We derive a formula comparing ordinary self-extension groups under graded twisting, apply the resulting criterion to the Liu--Schulz algebras, and describe the successive twists explicitly in the four-generator case.
\end{abstract}
\noindent\textbf{2020 Mathematics Subject Classification.}
Primary 16E05, 16E30; Secondary 16G10, 16W50.
\par\smallskip
\noindent\textbf{Keywords.}
Quantum complete intersection; self-extension; Auslander--Reiten conjecture;
graded twist; Liu--Schulz algebra.

\section{Introduction}

The Auslander--Reiten conjecture, formulated in the study of the
generalized Nakayama conjecture \cite{AR75}, asserts that an Artin algebra
\(A\) satisfies
\[
 \Ext_A^i(M,M\oplus A)=0\quad(i>0)
 \quad\Longrightarrow\quad M\text{ is projective}
\]
for every finitely generated left \(A\)-module \(M\). We abbreviate this
assertion to ARC. For a self-injective algebra, the groups
\(\Ext_A^i(M,A)\) vanish automatically for \(i>0\); ARC is then the
assertion that every self-orthogonal module is projective, usually called
Tachikawa's second conjecture, or TC2 \cite{Tachikawa73}.

Over commutative complete intersections, Avramov and Buchweitz proved
that vanishing of a single positive even self-extension group detects
finite projective dimension \cite[Theorem 4.2]{AB00}. For quantum complete
intersections, roots of unity occur as hypotheses in established
Ext-symmetry results \cite[Theorems 3.3 and 3.4]{Bergh09}; the roots of
unity case is also recorded among the classes satisfying ARC in
\cite[Introduction]{CHQW23}.

Fix a field \(k\), an integer \(c\geq1\), integers \(a_i\geq2\), and
parameters \(q_{ij}\in k^\times\) for \(i<j\). We consider the quantum
complete intersection
\begin{equation}\label{eq:QCI}
 A=A(\mathbf a,\mathbf q)=
 \frac{k\langle x_1,\ldots,x_c\rangle}
 {(x_i^{a_i}\ (1\leq i\leq c),\quad
   x_ix_j-q_{ij}x_jx_i\ (1\leq i<j\leq c))}.
\end{equation}
Our main result is the following projectivity criterion.

\begin{theoremA}
For every field \(k\), every algebra \(A\) in \eqref{eq:QCI}, and every
finite-dimensional left \(A\)-module \(M\),
\[
 \Ext_A^1(M,M)=0=\Ext_A^2(M,M)
 \quad\Longrightarrow\quad M\text{ is projective}.
\]
\end{theoremA}

Thus ARC and TC2 hold for \eqref{eq:QCI} with arbitrary nonzero
commutation parameters. In particular, they hold for the Liu--Schulz
family. The proof uses the gradability theorem of Amiot--Oppermann
\cite[Corollary 4.5]{AO13}, the graded twists introduced by Zhang
\cite{Zhang96}, and the commutative criterion of
\cite[Theorem 4.2]{AB00}. We construct a finite sequence of graded twists leading to a commutative truncated polynomial algebra. At each step, the first and second ordinary self-extension groups remain zero, and projectivity is preserved and reflected.

Section~\ref{sec:prelim} records the algebra and homological facts used
below. Section~\ref{sec:transport} establishes the twist and Ext formulas.
Section~\ref{sec:main} proves Theorem A, gives its consequences, and
contains a complete four-generator calculation.

\medskip\noindent\textbf{Conventions.}
Algebras and modules over \(k\) are finite-dimensional unless stated
otherwise, and all modules are left modules. Write \(\umod{B}\) for
ordinary modules and \(\gmod{B}\) for \(\Z\)-graded modules with
degree-zero morphisms; \(\Ext_B\) denotes ordinary Ext. Morphisms act on
the right and compose from left to right:
\((x)(fg)=((x)f)g\). Functors compose from right to left:
\((FG)(X)=F(G(X))\).

\section{Preliminaries}\label{sec:prelim}

\subsection{Quantum complete intersections}

For \(A\) in \eqref{eq:QCI}, set
\[
 I=\{\alpha\in\N^c:0\leq\alpha_i<a_i\},\qquad
 x^\alpha=x_1^{\alpha_1}\cdots x_c^{\alpha_c},\qquad
 \kappa(\alpha,\beta)=\prod_{u<v}q_{uv}^{-\alpha_v\beta_u}.
\]
All inequalities between exponent vectors are coordinatewise.

\begin{lemma}\label{lem:structure}
The elements \(x^\alpha\), \(\alpha\in I\), form a basis of \(A\), with
\begin{equation}\label{eq:PBWproduct}
 x^\alpha x^\beta=
 \begin{cases}
 \kappa(\alpha,\beta)x^{\alpha+\beta},&\alpha+\beta\in I,\\
 0,&\alpha+\beta\notin I.
 \end{cases}
\end{equation}
The algebra \(A\) is local and Frobenius, hence self-injective.
For each \(j\), it has the coordinate grading
\(\deg_jx_i=\delta_{ij}\). Every diagonal assignment
\((x_i)\delta_\lambda=\lambda_i x_i\), with \(\lambda_i\in k^\times\),
defines an automorphism preserving these gradings.
\end{lemma}

\begin{proof}
Reordering by \(x_vx_u=q_{uv}^{-1}x_ux_v\) shows that the displayed
monomials span and gives \eqref{eq:PBWproduct}. On a vector space with
basis indexed by \(I\), this formula defines an associative multiplication:
\[
 \kappa(\alpha,\beta)\kappa(\alpha+\beta,\gamma)
 =\kappa(\beta,\gamma)\kappa(\alpha,\beta+\gamma),
\]
If \(\alpha+\beta+\gamma\notin I\), both bracketings are zero.
The basis elements corresponding to the standard unit vectors
satisfy the defining relations of \(A\).
The resulting homomorphism from \(A\) sends the spanning monomials
to distinct basis elements, proving their linear independence.

The ideal \(J=(x_1,\ldots,x_c)\) is nilpotent and \(A/J=k\), so \(A\)
is local. Let \(\omega=(a_1-1,\ldots,a_c-1)\), and let \(\ell:A\to k\)
extract the coefficient of \(x^\omega\). Relative to the bases
\(\{x^\alpha\}\) and \(\{x^{\omega-\beta}\}\), the pairing
\((a,b)\mapsto(ab)\ell\) has diagonal matrix with nonzero entries
\(\kappa(\alpha,\omega-\alpha)\). It is associative and nondegenerate,
so \(A\) is Frobenius. Finally, the relations are homogeneous for the
coordinate gradings and are preserved by diagonal scaling.
\end{proof}

\subsection{Graded homological facts and scalar extension}

For \(U\in\gmod{B}\), use the shift \(U(d)_r=U_{r+d}\).
Thus a degree-\(d\) map \(U\to V\) is a degree-zero map \(U\to V(d)\).

\begin{lemma}\label{lem:graded-facts}
Let \(B\) be a finite-dimensional \(\Z\)-graded algebra and
\(U,V\in\gmod{B}\).
\begin{enumerate}
\item Every \(B\)-linear map is a finite sum of homogeneous maps, and
\[
 \Hom_B(U,V)=\bigoplus_{d\in\Z}\Hom_{\gmod{B}}(U,V(d)).
\]
\item The category \(\gmod{B}\) has enough projectives. Its projective
objects are precisely the objects whose underlying modules are
projective. Every object has a projective resolution with
finite-dimensional graded terms and degree-zero differentials.
\item For every \(n\geq0\),
\begin{equation}\label{eq:Ext-decomp}
 \Ext_B^n(U,V)\simeq
 \bigoplus_{d\in\Z}\Ext_{\gmod{B}}^n(U,V(d)).
\end{equation}
\end{enumerate}
\end{lemma}

\begin{proof}
The Hom decomposition and the projectivity comparison are
\cite[Theorem 1.2.6 and Proposition 1.2.15]{Hazrat16}, applied to
\(B^{\mathrm{op}}\) to obtain left-module statements. Shifted free
modules provide enough projectives; choosing finite homogeneous
generators for successive kernels gives the stated resolutions.
For (3), take such a resolution
\(\cdots\to P_1\to P_0\to U\to0\). It is also an ordinary projective
resolution. Apply (1) termwise to its Hom complex. The differential
preserves internal degree, and direct sums are exact, so taking
cohomology gives \eqref{eq:Ext-decomp}.
\end{proof}

\begin{lemma}\label{lem:base-change}
Let \(A\) be a finite-dimensional \(k\)-algebra, \(K/k\) a field
extension, and \(M,N\in\umod{A}\). Write \(A_K=K\otimes_k A\) and
\(M_K=K\otimes_kM\). Then
\begin{equation}\label{eq:base-Ext}
 K\otimes_k\Ext_A^n(M,N)\simeq\Ext_{A_K}^n(M_K,N_K)
 \qquad(n\geq0).
\end{equation}
Moreover, \(M_K\) is projective if and only if \(M\) is projective.
\end{lemma}

\begin{proof}
Take a projective resolution of \(M\) with finitely generated terms.
The degreewise Hom--scalar-extension isomorphisms and exactness of
\(K\otimes_k-\) give \eqref{eq:base-Ext}. Scalar extension preserves
projectivity. Conversely, suppose that \(M_K\) is projective.
Choose a short exact sequence
\[
 \xi:0\longrightarrow N\longrightarrow P
 \longrightarrow M\longrightarrow0
\]
with \(P\) finite free and apply \(K\otimes_k-\).
The resulting sequence splits because \(M_K\) is projective.
Thus \(1\otimes[\xi]=0\) by \eqref{eq:base-Ext}.
The map
\[
 V\longrightarrow K\otimes_k V,
 \qquad v\longmapsto1\otimes v,
\]
is injective for every \(k\)-vector space \(V\).
Hence \([\xi]=0\), so \(M\) is projective.
\end{proof}

\section{Gradings, twists, and ordinary self-extensions}\label{sec:transport}

Starting with an ordinary rigid module, we choose a compatible grading, apply a left Zhang twist, and then forget the grading. To compare ordinary self-extension groups before and after this construction, we express them as direct sums of graded Ext groups with shifted targets. Since twisting does not in general commute with grading shifts, we first establish the precise relation between these two operations. We then use this relation to prove the vanishing result needed in the proof of the main theorem.

A module \(M\) is \emph{rigid} if \(\Ext_B^1(M,M)=0\).
The following is the gradability theorem used in each step of the
construction.

\begin{theorem}[{\cite[Corollary 4.5]{AO13}}]\label{thm:AO}
Let \(k\) be algebraically closed and let \(B\) be a finite-dimensional
\(\Z\)-graded \(k\)-algebra. Every finite-dimensional rigid ordinary
left \(B\)-module admits a compatible grading.
\end{theorem}

The left-module statement follows by applying the cited result to
\(B^{\mathrm{op}}\). No semisimplicity assumption on \(B_0\) is needed.
For an automorphism \(\alpha\) of \(B\), let \(\tw{\alpha}{M}\) denote
the ordinary left module with the same vector space and action
\begin{equation*}
 b\cdot_\alpha m=((b)\alpha)m.
\end{equation*}
If \(f:M\to N\) is \(B\)-linear, the same underlying linear map
defines a \(B\)-linear map
\(\tw{\alpha}{M}\to\tw{\alpha}{N}\).
In particular, twisting by \(\alpha\) preserves isomorphisms.

\begin{proposition}\label{prop:diagonal}
Assume that \(k\) is algebraically closed. If \(M\) is a rigid module
over a quantum complete intersection \(B\), then
\(\tw{\delta}{M}\simeq M\) as ordinary modules for every diagonal
automorphism \(\delta\), where \((x_i)\delta=\lambda_i x_i\) and
\(\lambda_i\in k^\times\).
\end{proposition}

\begin{proof}
Fix a coordinate \(i\), give \(x_i\) degree one and the other generators
degree zero, and apply Theorem~\ref{thm:AO} to grade \(M\).
For \(t\in k^\times\), the automorphism \(\delta_{i,t}\) scaling only
\(x_i\) satisfies \((b)\delta_{i,t}=t^s b\) for \(b\in B_s\).
The linear isomorphism
\[
 h_t:M\longrightarrow\tw{\delta_{i,t}}{M},
 \qquad (m)h_t=t^r m\quad(m\in M_r)
\]
is \(B\)-linear, since for \(b\in B_s\),
\[
 (bm)h_t=t^{r+s}bm
 =((b)\delta_{i,t})((m)h_t).
\]
Every diagonal automorphism is a product of these commuting coordinate
scalings. Successively twisting the resulting ordinary isomorphisms
proves the assertion. The coordinate gradings may be chosen separately;
a simultaneous \(\Z^c\)-grading is not required.
\end{proof}

For the remainder of this section the field is arbitrary, \(B\) is a
finite-dimensional \(\Z\)-graded algebra, and \(\sigma\) is a
grading-preserving automorphism. We use the following left version of
the automorphism twist of \cite{Zhang96}. On the graded vector space
\(B\), set
\begin{equation*}
 a*b=((a)\sigma^t)b,
 \qquad a\in B_s,\ b\in B_t,
\end{equation*}
and denote the resulting algebra by \(B^\sigma\).
For \(M\in\gmod{B}\), let \(F_\sigma M\) have the same graded vector
space, with action
\begin{equation*}
 b*m=((b)\sigma^r)m,
 \qquad b\in B_s,\ m\in M_r.
\end{equation*}
Both formulas extend bilinearly. On degree-zero morphisms,
\(F_\sigma\) retains the underlying linear map.

\begin{lemma}\label{lem:twist-equivalence}
These formulas define a graded algebra and an exact equivalence
\[
 F_\sigma:\gmod{B}\longrightarrow\gmod{B^\sigma}
\]
with inverse \(F_{\sigma^{-1}}\). The underlying ordinary module of
\(F_\sigma M\) is projective if and only if that of \(M\) is projective.
\end{lemma}

\begin{proof}
For \(a\in B_u\), \(b\in B_s\), and \(c\in B_t\),
\[
 (a*b)*c=((a)\sigma^{s+t})((b)\sigma^t)c=a*(b*c).
\]
The unit is unchanged, and \(\sigma\) remains a graded automorphism
of \(B^\sigma\). For \(m\in M_r\),
\[
 a*(b*m)=((a)\sigma^{r+s})((b)\sigma^r)m=(a*b)*m.
\]
If \(f:M\to N\) is a degree-zero \(B\)-homomorphism, then
\((b*m)f=((b)\sigma^r)((m)f)=b*((m)f)\), so the construction is
functorial. Twisting again by \(\sigma^{-1}\) restores the original
products and actions:
\[
 ((a)\sigma^{-t})*b=ab,
 \qquad ((b)\sigma^{-r})*m=bm.
\]
Thus \(F_{\sigma^{-1}}F_\sigma=\Id\) and
\(F_\sigma F_{\sigma^{-1}}=\Id\), with the indicated source categories.
Exactness follows because the underlying vector spaces and maps are
unchanged. The equivalence preserves and reflects graded projectivity;
Lemma~\ref{lem:graded-facts} then gives the ordinary assertion.
\end{proof}

Recall the shift convention \(M(d)_r=M_{r+d}\). Since \(\sigma\)
preserves degrees, \(\tw{\sigma^d}{M}\) retains the grading of \(M\).

\begin{lemma}\label{lem:shift-twist}
For every \(d\in\Z\), there is an identity of graded
\(B^\sigma\)-modules
\begin{equation*}
 F_\sigma\bigl(\tw{\sigma^d}{M}(d)\bigr)=F_\sigma(M)(d).
\end{equation*}
\end{lemma}

\begin{proof}
The degree-\(r\) component on either side is \(M_{r+d}\).
For \(m\in M_{r+d}\) and \(b\in B_s\), the action on the left is
\[
 b*m=((b)\sigma^r)\cdot_{\sigma^d}m
     =((b)\sigma^{r+d})m.
\]
On the right, the twist is applied before the shift, while \(m\) has
degree \(r+d\), and gives the same action.
\end{proof}

\begin{proposition}\label{prop:twisted-Ext}
For \(M\in\gmod{B}\) and \(n\geq0\), there is a natural
vector-space isomorphism
\begin{equation}\label{eq:twisted-Ext}
 \Ext_{B^\sigma}^n(F_\sigma M,F_\sigma M)
 \simeq
 \bigoplus_{d\in\Z}
 \Ext_{\gmod{B}}^n\bigl(M,\tw{\sigma^d}{M}(d)\bigr).
\end{equation}
Here the Ext group on the left is computed in the ordinary module
category.
\end{proposition}

\begin{proof}
Apply the graded decomposition of ordinary Ext from
Lemma~\ref{lem:graded-facts}, followed by
Lemma~\ref{lem:shift-twist} and the exact equivalence above:
\begin{align*}
 \Ext_{B^\sigma}^n(F_\sigma M,F_\sigma M)
 &\simeq\bigoplus_d
 \Ext_{\gmod{B^\sigma}}^n(F_\sigma M,(F_\sigma M)(d))\\
 &\simeq\bigoplus_d
 \Ext_{\gmod{B^\sigma}}^n
 \bigl(F_\sigma M,F_\sigma(\tw{\sigma^d}{M}(d))\bigr)\\
 &\simeq\bigoplus_d
 \Ext_{\gmod{B}}^n\bigl(M,\tw{\sigma^d}{M}(d)\bigr).
\end{align*}
\end{proof}

\begin{proposition}\label{prop:transport}
Let \(M\in\gmod{B}\) and \(S\subseteq\Z_{>0}\). Suppose that
\[
 \Ext_B^n(M,M)=0\quad(n\in S),
 \qquad \tw{\sigma^d}{M}\simeq M\quad(d\in\Z)
\]
as ordinary modules. Then
\[
 \Ext_{B^\sigma}^n(F_\sigma M,F_\sigma M)=0\quad(n\in S).
\]
Furthermore, \(F_\sigma M\) is projective as an ordinary
\(B^\sigma\)-module if and only if \(M\) is projective as an
ordinary \(B\)-module.
\end{proposition}

\begin{proof}
Fix \(d\in\Z\) and \(n\in S\). The assumed ordinary isomorphism
and Lemma~\ref{lem:graded-facts} give
\[
 0=\Ext_B^n(M,\tw{\sigma^d}{M})
 \simeq\bigoplus_{e\in\Z}
 \Ext_{\gmod{B}}^n\bigl(M,\tw{\sigma^d}{M}(e)\bigr).
\]
In particular, its \(e=d\) summand is zero. This holds for every
\(d\), so \eqref{eq:twisted-Ext} proves the required vanishing.
The projectivity assertion is Lemma~\ref{lem:twist-equivalence}.
\end{proof}

The isomorphisms \(\tw{\sigma^d}{M}\simeq M\) need not preserve
the chosen gradings. Together with the assumed vanishing of
\(\Ext_B^n(M,M)\), they imply that
\(\Ext_B^n(M,\tw{\sigma^d}{M})=0\).
Each graded summand therefore vanishes.
This argument requires neither an equivalence of ordinary module
categories nor compatibility between independently chosen gradings.

\section{The main theorem and applications}\label{sec:main}

\subsection{Proof of the main theorem and its consequences}

The following lemma is the Artin local case of
Avramov--Buchweitz's self-extension theorem.

\begin{lemma}[{\cite[Theorem 4.2]{AB00}}]\label{lem:commutative}
Let \(C=k[t_1,\ldots,t_c]/(t_1^{a_1},\ldots,t_c^{a_c})\), where
\(a_i\geq2\). A finite-dimensional left \(C\)-module \(L\) satisfying
\(\Ext_C^2(L,L)=0\) is free.
\end{lemma}

\begin{proof}
The ring \(C\simeq k[[t_1,\ldots,t_c]]/(t_1^{a_1},\ldots,t_c^{a_c})\)
is an Artin local complete intersection. The cited theorem gives
\(\pd_C L<\infty\). If \(L\ne0\), the Auslander--Buchsbaum formula yields
\(\pd_C L=\depth C-\depth_C L=0\). Thus \(L\) is projective, hence free
over the local ring \(C\). The zero module is free as well.
\end{proof}

\begin{theorem}\label{thm:main}
Let \(k\) be any field and let \(A=A(\mathbf a,\mathbf q)\) be as in
\eqref{eq:QCI}. For every finite-dimensional left \(A\)-module \(M\),
\[
 \Ext_A^1(M,M)=0=\Ext_A^2(M,M)
 \quad\Longrightarrow\quad M\text{ is projective}.
\]
\end{theorem}

\begin{proof}
We first assume that \(k\) is algebraically closed.

\smallskip\noindent
\emph{The intermediate algebras.}
For \(1\leq j\leq c\), put
\begin{equation*}
 q_{uv}^{(j)}=
 \begin{cases}q_{uv},&v\leq j,\\1,&v>j,\end{cases}
 \qquad
 B_j=\frac{k\langle x_1,\ldots,x_c\rangle}
 {(x_i^{a_i},\ x_ux_v-q_{uv}^{(j)}x_vx_u\ (u<v))}.
\end{equation*}
Here all power relations are included. Then \(B_c=A\), every
\(x_\ell\) with \(\ell>j\) is central in \(B_j\), and
\[
 B_1=C=k[x_1,\ldots,x_c]/(x_1^{a_1},\ldots,x_c^{a_c}).
\]
Each \(B_j\) has the basis of Lemma~\ref{lem:structure}.
For \(j\geq2\), use \(\deg_jx_i=\delta_{ij}\) and set
\begin{equation*}
 (x_i)\sigma_j=
 \begin{cases}q_{ij}^{-1}x_i,&i<j,\\x_i,&i\geq j.\end{cases}
\end{equation*}
This is a graded diagonal automorphism. Writing \(*_j\) for the
twisted multiplication, we have, for \(i<j\),
\begin{equation*}
 x_i*_jx_j=q_{ij}^{-1}x_ix_j=x_jx_i=x_j*_jx_i.
\end{equation*}
Two generators different from \(x_j\) both have degree zero, so their
products are unchanged. If \(\ell>j\), then \(\sigma_j\) fixes
\(x_\ell\), and \(x_\ell\) still commutes with \(x_j\).
All power relations are preserved, since a generator different from
\(x_j\) has degree zero and \(\sigma_j\) fixes \(x_j\).
Consequently there is a graded homomorphism
\begin{equation*}
 \phi_j:B_{j-1}\longrightarrow B_j^{\sigma_j},\qquad
 (x_i)\phi_j=x_i,
\end{equation*}
where both sides carry the \(j\)-th coordinate grading. On ordered
monomials it is given by
\begin{equation*}
 (x^\alpha)\phi_j=\lambda_j(\alpha)x^\alpha,
 \qquad \lambda_j(\alpha)=\prod_{i<j}q_{ij}^{-\alpha_i\alpha_j}.
\end{equation*}
Indeed, each of the \(\alpha_j\) occurrences of \(x_j\), when multiplied
in on the right, contributes \(\prod_{i<j}q_{ij}^{-\alpha_i}\).
All these coefficients are nonzero, so \(\phi_j\) maps the standard
basis to a basis and is an isomorphism.

\smallskip\noindent
\emph{The modules and the induction.}
Start with \(M_c=M\). Suppose \(M_j\in\umod{B_j}\) has been constructed,
with both \(\Ext_{B_j}^1(M_j,M_j)\) and \(\Ext_{B_j}^2(M_j,M_j)\) zero.
Theorem~\ref{thm:AO} provides a compatible \(j\)-th coordinate grading
\[
 \widetilde M_j=\bigoplus_{r\in\Z}(M_j)_r^{[j]}.
\]
Let \(R_j:\gmod{B_j^{\sigma_j}}\to\gmod{B_{j-1}}\) be restriction along
\(\phi_j\). Define \(M_{j-1}\) to be the ordinary module underlying
\(R_jF_{\sigma_j}(\widetilde M_j)\). For \(m\in(M_j)_r^{[j]}\),
\begin{equation*}
 x_i\cdot_{j-1}m=
 \begin{cases}
 q_{ij}^{-r}(x_i\cdot_jm),&i<j,\\
 x_i\cdot_jm,&i\geq j.
 \end{cases}
\end{equation*}
For an ordered monomial the full formula, including \(\phi_j\), is
\begin{equation*}
 x^\alpha\cdot_{j-1}m=
 \left(\prod_{i<j}q_{ij}^{-\alpha_i(r+\alpha_j)}\right)
 (x^\alpha\cdot_jm),
 \qquad m\in(M_j)_r^{[j]}.
\end{equation*}

Every power \(\sigma_j^d\) is diagonal. Proposition~\ref{prop:diagonal}
therefore gives \(\tw{\sigma_j^d}{M_j}\simeq M_j\) as ordinary modules
for all \(d\in\Z\). Proposition~\ref{prop:transport}, with
\(S=\{1,2\}\), and restriction along \(\phi_j\) now yield
\begin{gather}\label{eq:induction-invariants}
 \dim_kM_{j-1}=\dim_kM_j,\qquad
 \Ext_{B_{j-1}}^n(M_{j-1},M_{j-1})=0\quad(n=1,2),\\
 M_{j-1}\text{ is projective}
 \quad\Longleftrightarrow\quad M_j\text{ is projective}.\notag
\end{gather}
In particular, the new ordinary module \(M_{j-1}\) is rigid. After
forgetting its current grading, Theorem~\ref{thm:AO} can be applied
again for the \((j-1)\)-st coordinate grading. No compatibility between
successive choices is required. This proves that the induction is
legitimate at every stage.

After \(c-1\) steps, \(M_1\) is a \(C\)-module with
\(\Ext_C^2(M_1,M_1)=0\). Lemma~\ref{lem:commutative} makes \(M_1\)
free. Reflecting projectivity through \eqref{eq:induction-invariants}
gives projectivity of \(M_c=M\). If \(c=1\), there are no twists and
Lemma~\ref{lem:commutative} applies directly.

\smallskip\noindent
\emph{Descent to the original field.}
For arbitrary \(k\), choose an algebraic closure \(K\). The algebra
\(A_K=K\otimes_kA\) is again the quantum complete intersection
\eqref{eq:QCI} over \(K\). By Lemma~\ref{lem:base-change}, \(M_K\)
has zero first and second self-extension groups. The algebraically closed case shows that \(M_K\) is projective.
Lemma~\ref{lem:base-change} then implies that \(M\) is projective.
\end{proof}

\begin{corollary}\label{cor:ARC}
Every quantum complete intersection \eqref{eq:QCI} satisfies ARC and TC2.
\end{corollary}

\begin{proof}
Self-orthogonality implies the two vanishings in Theorem~\ref{thm:main}.
The algebra is self-injective by Lemma~\ref{lem:structure}, so this
proves TC2 and ARC.
\end{proof}

We now apply the main theorem to the Liu--Schulz algebras.
In the classical construction, the parameter \(q\) is assumed
to have infinite multiplicative order
\cite[Section 6]{CPX12}.

\begin{corollary}\label{cor:LS}
For every field \(k\) and \(q\in k^\times\), let
\[
 \Lambda_q=
 \frac{k\langle x,y,z\rangle}
 {(x^2,y^2,z^2,\ yx+qxy,\ zy+qyz,\ xz+qzx)}.
\]
Then \(\Lambda_q\) is an eight-dimensional local symmetric algebra, and
every finite-dimensional left \(\Lambda_q\)-module \(M\) satisfies
\[
 \Ext_{\Lambda_q}^1(M,M)=0=\Ext_{\Lambda_q}^2(M,M)
 \quad\Longrightarrow\quad M\text{ is projective}.
\]
In particular, \(\Lambda_q\) satisfies ARC and TC2.
\end{corollary}

\begin{proof}
This is \eqref{eq:QCI} with
\[
 (x_1,x_2,x_3)=(x,y,z),\quad (a_1,a_2,a_3)=(2,2,2),\quad
 (q_{12},q_{13},q_{23})=(-q^{-1},-q,-q^{-1}).
\]
The projectivity assertion and its consequences follow from
Theorem~\ref{thm:main} and Corollary~\ref{cor:ARC}.
Lemma~\ref{lem:structure} gives locality and the basis
\(1,x,y,z,xy,xz,yz,xyz\).
For symmetry, let \(\ell\) extract the coefficient of \(xyz\).
The relations give
\[
 zxy=xyz=yzx,\qquad (xz)y=-qxyz=y(xz).
\]
Hence degree-two elements commute with the generators;
products with a repeated generator are zero.
For homogeneous elements \(a\) and \(b\), the Frobenius pairing
\((a,b)\mapsto(ab)\ell\) vanishes unless
\(\deg a+\deg b=3\).
In the remaining cases, symmetry follows from the commutation
relations above and the fact that degree-zero elements are scalars.
Thus the nondegenerate Frobenius pairing is symmetric.
\end{proof}

\subsection{A four-generator calculation}\label{subsec:four}

Let \(c=4\). The relations \(x_i^{a_i}=0\) remain unchanged throughout.
The following table lists the commutation parameters in the order
\((12,13,14,23,24,34)\).
In the last column, each tuple
\((\lambda_1,\lambda_2,\lambda_3,\lambda_4)\)
denotes the diagonal automorphism given by
\((x_i)\sigma=\lambda_i x_i\).
\[
\begin{array}{c|c|c}
 &\text{commutation parameters}&\text{next automorphism}\\ \hline
 B_4&(q_{12},q_{13},q_{14},q_{23},q_{24},q_{34})
 &(q_{14}^{-1},q_{24}^{-1},q_{34}^{-1},1)\\
 B_3&(q_{12},q_{13},1,q_{23},1,1)&(q_{13}^{-1},q_{23}^{-1},1,1)\\
 B_2&(q_{12},1,1,1,1,1)&(q_{12}^{-1},1,1,1)\\
 B_1&(1,1,1,1,1,1)&\text{none}.
\end{array}
\]
Assume first that \(k\) is algebraically closed
and \(M_4=M\) satisfies \(\Ext_{B_4}^n(M_4,M_4)=0\) for \(n=1,2\).

\smallskip\noindent
\emph{First step: \(B_4\) to \(B_3\).}
Use \(\deg_4x_4=1\) and \(\deg_4x_i=0\) for \(i<4\), and let
\[
 (x_i)\sigma_4=q_{i4}^{-1}x_i\quad(i=1,2,3),\qquad
 (x_4)\sigma_4=x_4.
\]
Then
\[
 x_i*_4x_4=q_{i4}^{-1}x_ix_4=x_4x_i=x_4*_4x_i\quad(i=1,2,3),
\]
while the relations among \(x_1,x_2,x_3\) are unchanged. The isomorphism
\(\phi_4:B_3\to B_4^{\sigma_4}\) satisfies
\[
 (x^\alpha)\phi_4=
 q_{14}^{-\alpha_1\alpha_4}q_{24}^{-\alpha_2\alpha_4}
 q_{34}^{-\alpha_3\alpha_4}x^\alpha.
\]
Choose a compatible grading of \(M_4\). For
\(m\in(M_4)_r^{[4]}\), the action on \(M_3\) is
\[
\begin{aligned}
 x_1\cdot_3m&=q_{14}^{-r}(x_1\cdot_4m),&
 x_2\cdot_3m&=q_{24}^{-r}(x_2\cdot_4m),\\
 x_3\cdot_3m&=q_{34}^{-r}(x_3\cdot_4m),&
 x_4\cdot_3m&=x_4\cdot_4m.
\end{aligned}
\]
By Proposition~\ref{prop:transport}, both self-extension groups vanish
for \(M_3\), and \(M_3\) is projective if and only if \(M_4\) is.

\smallskip\noindent
\emph{Second step: \(B_3\) to \(B_2\).}
Forget the preceding grading of \(M_3\), give \(x_3\) degree one in
\(B_3\), and choose a compatible grading of the rigid module \(M_3\).
Set
\[
 (x_1)\sigma_3=q_{13}^{-1}x_1,\quad
 (x_2)\sigma_3=q_{23}^{-1}x_2,\quad
 (x_3)\sigma_3=x_3,\quad (x_4)\sigma_3=x_4.
\]
For \(i=1,2\),
\[
 x_i*_3x_3=q_{i3}^{-1}x_ix_3=x_3x_i=x_3*_3x_i.
\]
Also \(x_3*_3x_4=x_3x_4=x_4x_3=x_4*_3x_3\), and \(x_4\) still
commutes with \(x_1,x_2\). Only \(q_{12}\) can remain nontrivial.
The isomorphism \(\phi_3:B_2\to B_3^{\sigma_3}\) has formula
\[
 (x^\alpha)\phi_3=
 q_{13}^{-\alpha_1\alpha_3}q_{23}^{-\alpha_2\alpha_3}x^\alpha.
\]
For \(m\in(M_3)_s^{[3]}\), define
\[
\begin{aligned}
 x_1\cdot_2m&=q_{13}^{-s}(x_1\cdot_3m),&
 x_2\cdot_2m&=q_{23}^{-s}(x_2\cdot_3m),\\
 x_3\cdot_2m&=x_3\cdot_3m,&
 x_4\cdot_2m&=x_4\cdot_3m.
\end{aligned}
\]
Again the two self-extension groups vanish, and projectivity is
preserved and reflected. The degree \(s\) belongs to the newly chosen
grading; a vector homogeneous for the previous grading need not be
homogeneous for this one.

\smallskip\noindent
\emph{Third step: \(B_2\) to \(B_1\).}
Give \(x_2\) degree one, choose a compatible grading of the rigid
ordinary module \(M_2\), and put
\[
 (x_1)\sigma_2=q_{12}^{-1}x_1,\qquad
 (x_i)\sigma_2=x_i\quad(i=2,3,4).
\]
Then
\[
 x_1*_2x_2=q_{12}^{-1}x_1x_2=x_2x_1=x_2*_2x_1.
\]
The generators \(x_3,x_4\) stay central, and therefore
\[
 B_1=k[x_1,x_2,x_3,x_4]/(x_1^{a_1},x_2^{a_2},x_3^{a_3},x_4^{a_4}),
 \qquad (x^\alpha)\phi_2=q_{12}^{-\alpha_1\alpha_2}x^\alpha.
\]
For \(m\in(M_2)_t^{[2]}\), the final action is
\[
 x_1\cdot_1m=q_{12}^{-t}(x_1\cdot_2m),\qquad
 x_i\cdot_1m=x_i\cdot_2m\quad(i=2,3,4).
\]
The module \(M_1\) has \(\Ext_{B_1}^2(M_1,M_1)=0\), so
By Lemma~\ref{lem:commutative}, \(M_1\) is free. Finally,
\[
 M_1\text{ projective}\Longrightarrow M_2\text{ projective}
 \Longrightarrow M_3\text{ projective}\Longrightarrow M_4=M\text{ projective}.
\]
The construction uses three separately chosen gradings, three twists,
and all six original parameters. Over an arbitrary field, it is
performed after extension to an algebraic closure and followed by
Lemma~\ref{lem:base-change}.

\end{document}